\documentclass[aap]{imsart}

\RequirePackage{amsthm,amsmath,amsfonts,amssymb}
\RequirePackage[numbers]{natbib}
\RequirePackage[colorlinks,citecolor=blue,urlcolor=blue]{hyperref}
\RequirePackage{graphicx}
\usepackage{mathrsfs}
\usepackage{enumerate}
\usepackage{enumitem}

\newcommand{\tree}[2]{\begin{array}{c} #1 \\[-3pt] #2 \end{array}}
\newcommand{\btree}[3]{\begin{array}{c} #1 \\[-3pt] #2 \\[-3pt] #3 \end{array}}
\newcommand{\ttree}[4]{\begin{array}{c} #1 \\[-3pt] #2 \\[-3pt] #3 \\[-3pt] #4 \end{array}}
\newcommand{\block}[1]{
  \underbrace{\begin{matrix}2 & 1 & 2 & 1 & \cdots \end{matrix}}_{#1}
}
\newcommand{\blockb}[1]{
  \underbrace{\begin{matrix}1212 \cdots \end{matrix}}_{#1}
}

\startlocaldefs
\theoremstyle{plain}

\newtheorem{theorem}{Theorem}[section]
\newtheorem{lemma}[theorem]{Lemma}
\newtheorem{corollary}[theorem]{Corollary}
\newtheorem{proposition}[theorem]{Proposition}
\newtheorem{remark}[theorem]{Remark}
\theoremstyle{remark}
\newtheorem{definition}[theorem]{Definition}

\endlocaldefs

\begin{document}

\begin{frontmatter}
\title{One-dependent colorings of the star graph}
\runtitle{One-dependent graphs}

\begin{aug}
\author[A]{\fnms{Thomas M.} \snm{Liggett}\ead[label=e1]{tml@math.ucla.edu}},
\and
\author[B]{\fnms{Wenpin} \snm{Tang}\ead[label=e3]{wt2319@columbia.edu}}
\address[A]{Department,
Department of Mathematics, UCLA,
\printead{e1}}

\address[B]{Department of Industrial Engineering and Operations Research, Columbia University,
\printead{e3}}
\end{aug}

\begin{abstract}
This paper is concerned with symmetric $1$-dependent colorings of the $d$-ray star graph $\mathscr{S}^d$ for $d \ge 2$.
We compute the critical point of the $1$-dependent hard-core processes on $\mathscr{S}^d$, which gives a lower bound for the number of colors needed for a $1$-dependent coloring of $\mathscr{S}^d$.
We provide an explicit construction of a $1$-dependent $q$-coloring for any $q \ge 5$ of the infinite subgraph $\mathscr{S}^3_{(1,1,\infty)}$, which is symmetric in the colors and whose restriction to any path is some symmetric $1$-dependent $q$-coloring.
We also prove that there is no such coloring of $\mathscr{S}^3_{(1,1,\infty)}$ with $q = 4$ colors.
A list of open problems are presented.
\end{abstract}

\begin{keyword}[class=MSC]
\kwd[Primary ]{60C05}
\kwd{05C15}
\kwd[; secondary ]{82B27}
\end{keyword}

\begin{keyword}
\kwd{Hard-core processes}
\kwd{random colorings}
\kwd{one-dependent processes}
\kwd{star graph}
\end{keyword}

\end{frontmatter}

\section{Introduction and main results}

Finitely dependent processes have received considerable interest in probability theory and dynamical systems. 
They have a variety of applications including 
economic theory \cite{Ibr05, Ibr09}, statistics \cite{AP14, GNW01}, queueing systems \cite{HG01, Sigman90}, stochastic geometry \cite{Hein13}, and biology \cite{FRT20, MKS16}.
Early work of Ibragimov and Linnik \cite{IL71} explored a representation of finitely dependent processes via block-factors.
A block-factor can be expressed as a finite-range function of a family of independent and identically distributed random variables.
It is straightforward that a block-factor is stationary finitely dependent, and they showed that the converse is true for Gaussian processes. 
This left the question of whether any stationary finitely dependent process is a block-factor for decades.
It was not until the late 1980s that Aaronson, Gilat, Keane and de Valk \cite{AGKK} discovered stationary $1$-dependent processes which are not $2$-block-factors.
There have now been a rich body of literature on finitely (and especially one-) dependent processes, 
with topics such as classical limit theory, examples in connection with integrable probability and combinatorial interpretations; 
see \cite{BDF, HL16} for an overview.

The constructions in \cite{AGKK} and many subsequent works are purely technical and counter-intuitive.
An important question is whether `natural' examples of  finitely dependent processes are all block-factors.
In fact, several works \cite{GH89, Svante84, Svante15} assumed a block-factor assumption in the study of finitely dependent processes: if natural finitely dependent processes are block-factors, then there is little to be lost by making this assumption.
In their groundbreaking work \cite{HL16}, Holroyd and Liggett considered stationary finitely dependent proper colorings,
providing the first natural finitely dependent processes that are not any block-factors.
Proper coloring has applications in computer science. 
Colors may represent time schedules for job tasks in a network, where any adjacent pair of job tasks are not permitted to conflict with each other.
Finite dependence introduces a security benefit -- an adversary can get access to information of some job tasks only within a fixed finite range.
To proceed further, we need a few vocabularies.
\begin{definition}
Let $G = (V, E)$ be a simple graph.
A stochastic process $X = (X_v)_{v \in V}$ indexed by the set of vertices $V$ is
\begin{itemize}[itemsep = 3 pt]
\item
a {\bf (proper) $q$-coloring} if each $X_v$ takes values in $[q]: = \{1,\ldots,q\}$ and almost surely $X_u \neq X_v$ whenever $u$ and $v$ are neighbors;
\item
{\bf $k$-dependent} if its restriction to two subsets of $V$ are independent whenever these subsets are at graph distance larger than $k$ from each other.
\end{itemize}
The process $X$ is finitely dependent if it is $k$-dependent for some $k \ge 1$, and it is a coloring if it is a $q$-coloring for some $q \ge 2$.
\end{definition}

In the last five years, there has been much effort in understanding stationary finitely dependent colorings of the integers, i.e. $V = \mathbb{Z}$ and $G$ is the infinite $2$-regular tree.
It is obvious that stationary finitely dependent $2$-coloring of $\mathbb{Z}$ fails to exist.
Holroyd, Schramm and Wilson \cite{HSW} showed that there is no stationary $1$-dependent $3$-coloring of $\mathbb{Z}$.
Holroyd and Liggett \cite{HL16} found a stationary $1$-dependent $4$-coloring and a stationary $2$-dependent $3$-coloring of $\mathbb{Z}$, 
which imply the existence of stationary $k$-dependent $q$-colorings for all $k \ge 1$ and $q \ge 3$ except $(k, q) = (1,3)$ by splitting a color into further colors using external randomness.
However, such constructed colorings are not symmetric under permutations of the colors for a general $q$.
Further in \cite{HL15}, they constructed for each $q \ge 4$ a stationary $1$-dependent $q$-coloring
which is invariant in law under permutations of the colors and under reflection.
Precisely, they provided the following remarkable deletion-concatenation recursion for the cylinder probability $P^*$ of the $q$-coloring ${\bf x} \in [q]^n$ (which represents a sequence of consecutive colors along a path):
\begin{equation}
\label{eq:DCR1} 
P^*({\bf x}) : = \frac{1}{D(n+1)} \sum_{j = 1}^n C(n- 2j + 1) P^*(\widehat{\bf x}_j),
\end{equation}
where $\widehat{\bf x}_j$ is obtained by deleting the $j^{th}$ entry of ${\bf x}$, 
and $C(n): = T_n(\sqrt{q}/2)$ for $n \ge 0$ and $D(n): = \sqrt{q}U_{n-1}(\sqrt{q}/2)$ for $n \ge 1$, 
with $T_n$ and $U_n$ the Chebyshev polynomials of the first and second kind respectively. 
More explicitly, 
\begin{subequations}
\label{eq:CD}
\begin{align}
&C(n) = \frac{1}{2} \left[ \left(\frac{\sqrt{q} + \sqrt{q - 4}}{2} \right)^n + \left(\frac{\sqrt{q} - \sqrt{q - 4}}{2} \right)^n \right]  & \mbox{for } n \ge 0, \\
& D(n) = \sqrt{\frac{q}{q-4}} \left[ \left(\frac{\sqrt{q} + \sqrt{q - 4}}{2} \right)^n - \left(\frac{\sqrt{q} - \sqrt{q - 4}}{2} \right)^n \right] & \mbox{for } n \ge 1.
\end{align}
\end{subequations}
Later Holroyd, Hutchcroft and Levy \cite{HHL20} gave a probabilistic construction of the $1$-dependent $q$-coloring defined by \eqref{eq:DCR1} along with other stationary finitely dependent colorings of $\mathbb{Z}$ using the Mallows permutations.
They also considered the problem of expressing a stationary finitely dependent coloring as a finitary factor of an independent and identically distributed sequence.
See \cite{GKZ19, HHL18} for other finitely dependent coloring models, 
and \cite{H17, Spin19} for recent progress on finitary factors.  

As the readers may notice, recent works of finitely dependent colorings focused mostly on the integer case and its obvious extensions.
For instance, it is easy to use the construction \eqref{eq:DCR1} to build a $1$-dependent $4^d$-coloring of $\mathbb{Z}^d$, 
which is stationary but is not invariant under all isometries of $\mathbb{Z}^d$.
In fact, very little is known about `fully' symmetric (or automorphism invariant) finitely dependent colorings of 
graphs such as $\mathbb{Z}^d$ for $d \ge 2$ and homogeneous trees.
One main open problem is to find an automorphism invariant $1$-dependent coloring of $\mathbb{Z}^d$ for each $d \ge 2$,
or more generally, of the $d$-regular tree for each $d \ge 3$.
Note that any automorphism invariant coloring of $\mathbb{Z}^d$ or the $d$-regular tree (if it exists), 
when restricted to a copy of $\mathbb{Z}$,
is a coloring of $\mathbb{Z}$ that is invariant under permutations of the colors. 
Moreover, it was conjectured in \cite{HHL20, HL16} that the distribution \eqref{eq:DCR1} defines the unique $1$-dependent coloring which is invariant in law under permutations of the colors, under translation and reflection. 
So it is natural to look for a $1$-dependent coloring of an infinite graph such that
\begin{enumerate}[itemsep = 3 pt]
\item[$(i)$]
the coloring is invariant in law under permutations of the colors;
\item[$(ii)$]
the restriction to any path is distributed as \eqref{eq:DCR1}.
\end{enumerate}

As pointed out in \cite{HL16}, there is a close connection between proper colorings and hard-core processes. 
A hard-core process on $G$ is a process $(J_v)_{v \in V} \in \{0,1\}^V$ such that almost surely we do not have $J_u = J_v = 1$ whenever $u$ and $v$ are neighbors.
So if $X$ is a $q$-coloring of $G$, then $J_v : = 1_{\{X_v = x\}}$ defines a hard-core process for any given color $x \in [q]$; 
if $X$ is $1$-dependent, then so is $J$.
As an immediate consequence, the number of colors needed for a $1$-dependent coloring of $G$ is bounded from below by 
$1/p_h(G)$, where 
\begin{equation}
\label{critical}
p_h(G): = \sup\{p: \exists \mbox{ 1-dependent hard-core process $J$ with } \mathbb{P}(J_v = 1) = p \,\, \forall v\},
\end{equation}
is referred to as the critical point.
Motivated by the study of the repulsive lattice gas, 
Scott and Sokal \cite{SS05,SS06} showed that for any infinite connected graph $G$ of maximum degree $d$, 
\begin{equation}
\label{bounds}
\frac{(d-1)^{d-1}}{d^d} \le p_h(G) \le \frac{1}{4} \quad \mbox{for all } d \ge 2. 
\end{equation}
The lower bound in \eqref{bounds} is known to be achieved by the $d$-regular tree \cite{Sh}, 
which is a version of the Lov\'{a}sz local lemma.
So any (symmetric) $1$-dependent coloring of the $d$-regular tree requires at least $d^d/(d-1)^{d-1}$ colors, e.g. $7$ colors for the $3$-regular tree, and $10$ colors for the $4$-regular tree or $\mathbb{Z}^2$.
For $d = 2$, the lower bound for the number of colors needed for a symmetric $1$-dependent coloring of $\mathbb{Z}$ is achieved by the construction \eqref{eq:DCR1}.
But it remains unknown whether one can find an automorphism invariant $1$-dependent $q$-coloring for any $q$ of the $d$-regular tree with $d > 2$.

The purpose of this paper is to provide further study of symmetric $1$-dependent colorings of general graphs.
As we will see later, this does not seem to be a simple task even going slightly beyond the integer case.
Defining cylinder probabilities via recursion similar to \eqref{eq:DCR1} is arguably the most direct and promising way to construct any $1$-dependent coloring of a graph.
However, the idea of deletion-concatenation is not easily replicated if some vertices have degree greater than $2$,
and the question is how to genuinely concatenate its neighbors once a vertex is deleted.
In view of this, we consider for any $d \ge 3$ the $d$-ray star graph and its subgraphs, 
which are building blocks for general graphs.
Particular focus will be placed on the $3$-ray star graph.

The $d$-ray graph $\mathscr{S}^d$ has a distinguished vertex of degree $d$ and all the others of degree $2$,
and it consists of $d$ rays ${\bf v}^1 = (v_{11}, v_{12}, \ldots)$, ${\bf v}^2 = (v_{21}, v_{22}, \ldots), \ldots$, and ${\bf v}^d = (v_{d1}, v_{d2}, \ldots)$, emanating from the distinguished vertex $v_0$. 
We also identify the $2$-ray star graph $\mathscr{S}^2$ with the integers. 
See Figure \ref{fig:3rstar} below for a $3$-ray star graph.
\begin{figure}[h]
\label{fig:3rstar}
\includegraphics[width=0.42 \textwidth]{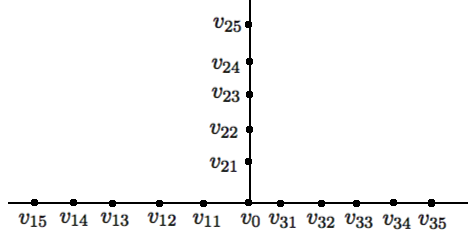}
\caption{The $3$-ray star graph}
\end{figure}
Let $\mathbb{N}: = \{1,2,\ldots\}$ be the set of positive integers.
For ${\bf n}: = (n_1,\ldots,n_d) \in \mathbb{N}^d$, let $\mathscr{S}^d_{\bf n}$ be the finite subgraph of the $d$-ray star graph, with $d$ rays of lengths $n_1, \ldots, n_d$ respectively.
Our first result gives an explicit formula for the critical point of $\mathscr{S}^d_{\bf n}$ for each $d \ge 2$ and ${\bf n} \in \mathbb{N}^d$, 
thus providing a lower bound for the number of colors needed for a symmetric $1$-dependent coloring of the $d$-ray star graph.

\begin{theorem}
\label{thm:main1}
For $d \ge 2$ and ${\bf n} \in \mathbb{N}^d$, the critical point $p_h(\mathscr{S}^d_{\bf n})$ is given by
\begin{equation}
\label{eq:phSS}
p_h(\mathscr{S}^d_{\bf n}) = \sup\left\{p: \, p \le \prod_{i = 1}^d \inf_{k \le n_i} a_k(p)\right\},
\end{equation}
where $(a_k(p); \, k \ge 1)$ is defined recursively by
\begin{equation}
\label{recsub}
a_1(p) = 1 - p, \quad \mbox{and} \quad a_k(p) + \frac{p}{a_{k-1}(p)} = 1 \mbox{ for } k \ge 2.
\end{equation}
Consequently, we have:
\begin{enumerate}[itemsep = 3 pt]
\item[(1)] 
The critical point $p_h(\mathscr{S}^d)$  
is the unique solution on $(0, 1/4]$ to the equation:
\begin{equation}
\label{eqph}
\left(\frac{1+ \sqrt{1-4p}}{2}\right)^d = p.
\end{equation}
\item[(2)]
Suppose that there exists a $1$-dependent $q$-coloring $X$ of the $d$-ray star graph $\mathscr{S}^d$ with $(X_v)_{v \in V}$ identically distributed. Then 
\begin{equation}
q \ge \frac{1}{p_h(\mathscr{S}^d)},
\end{equation}
where $p_h(\mathscr{S}^d)$ is given by \eqref{eqph}.
So there is no $1$-dependent $4$ coloring of the $3$-ray star graph $\mathscr{S}^3$,
which is invariant in law under permutations of the colors (i.e. satisfies (i)).
\end{enumerate}
\end{theorem}

\quad It is easy to derive from \eqref{eqph} that $p_h(\mathscr{S}^2) = 1/4$, $p_h(\mathscr{S}^3) = \sqrt{5} -2 $, 
$p_h(\mathscr{S}^4) = \frac{1}{3} \bigg(2 - 11$
$\left(\frac{2}{9 \sqrt{93} - 47} \right)^{\frac{1}{3}} + \left(\frac{9 \sqrt{93} - 47}{2} \right)^{\frac{1}{3}} \bigg)$,
and $p_h(\mathscr{S}^d) \sim  \ln d / d$ as $d \rightarrow \infty$.
Table $1$ below displays the numerical values of $p_h(\mathscr{S}^d)$ as $d$ ranges from $2$ to $11$.
\begin{center}
\begin{tabular}{ |c|c|c|c|c|c|c|c|c|c|c|} 
 \hline
 $d$ & 2 & 3 & 4 & 5 & 6 & 7 & 8 & 9 &10 & 11 \\ \hline
 $p_h(\mathscr{S}^d)$ & 0.250 & 0.236 & 0.217 & 0.199 & 0.185 & 0.173 & 0.162 & 0.153 & 0.149 & 0.138 \\
 \hline
\end{tabular}
 \\[5pt]
TABLE 1. Numerical values of $p_h(\mathscr{S}^d)$, $2 \le d \le 11$.
\end{center}
Theorem \ref{thm:main1} is essentially a consequence of \cite[Lemma 24]{HL16}, which is related to the smallest real zero of the partition function or the independence polynomial of the hard-core model. 
The explicit formula \eqref{eqph} for the critical point of $\mathscr{S}^d$ is new.
For ease of reference, we prove Theorem \ref{thm:main1} in Section \ref{sc3}.

Recall that the lower bound in \eqref{bounds} is achieved by the $d$-regular tree, 
and an interesting question is whether the upper bound in \eqref{bounds} for $d \ge 3$ can also be achieved by some subgraph of the $d$-regular tree with at least one vertex of degree $d$.
Since any infinite connected graph of maximum degree $d$ contains a copy of $\mathscr{S}^d_{(1,\ldots, 1, n)}$ for each $n$, we get the following corollary.

\begin{corollary}
For $d \ge 2$, let $G$ be an infinite connected graph of maximum degree $d$. 
Then
\begin{equation}
\label{newbounds}
\frac{(d-1)^{d-1}}{d^d} \leq p_h(G) \leq p_{\star}(d),
\end{equation}
where $p_{\star}(d) = \lim_{n \to \infty}p_h (\mathscr{S}^d_{(1,\ldots, 1, n)})$,
so
$p_{\star} (2) = p_{\star}(3) = 1/4$ and for $d \ge 4$, $p_{\star}(d)$ is the unique solution on $(0,1/4]$ to the equation:
\begin{equation}
\label{eq:110}
\frac{1}{2} (1-p)^{d-1} (1+\sqrt{1-4p}) = p.
\end{equation}
\end{corollary}

Note that the bounds in \eqref{newbounds} are tight, and the upper bound is improved for all $d \ge 4$.
The formula \eqref{eq:110} for $\lim_{n \to \infty}p_h (\mathscr{S}^d_{(1,\ldots, 1, n)})$ is derived similarly as
\eqref{eqph} for $p_h(\mathscr{S}^d)$ (see Section \ref{sc31}).
It is also easy to see that $(d-1)^{d-1}/d^d \sim 1/(ed)$ and $p_{\star}(d) \sim \ln d /d$ as $d \rightarrow \infty$.
Table $2$ below displays the numerical values of $p_{\star}(d)$ as $d$ ranges from $2$ to $11$.
\begin{center}
\begin{tabular}{ |c|c|c|c|c|c|c|c|c|c|c|} 
 \hline
 $d$ & 2 & 3 & 4 & 5 & 6 & 7 & 8 & 9 &10 & 11  \\ \hline
 $p_{\star}(d)$ & 0.250 & 0.250 & 0.245 & 0.229 & 0.212 & 0.197 & 0.183 & 0.172 & 0.162 & 0.153 \\
 \hline
\end{tabular}
\\[5pt]
TABLE 2. Numerical values of $p_{\star}(d)$, $2 \le d \le 11$.
\end{center}

Now we focus our attention on the $3$-ray star graph and its subgraphs.
Recall that our goal is to find a $1$-dependent coloring of the graph which satisfies the conditions $(i)$-$(ii)$.
According to Theorem \ref{thm:main1}, such a coloring requires at least $5$ colors.
But constructing such a $q$-coloring for any $q \ge 5$ still seems to be far-fetched.
A potentially less challenging task is to find a symmetric $1$-dependent coloring of the subgraph $\mathscr{S}^3_{(m,n,\infty)}$ for some finite $m,n$.
Again by Theorem \ref{thm:main1}, it is easy to get a lower bound for the number of colors required for the coloring, e.g.
\begin{itemize}[itemsep = 3 pt]
\item
for $(m, n) = (1,1)$ or $(1,2)$, the number of colors needed is at least $4$;
\item
for $m, n \ge 2$, the number of colors needed is bounded from below by $5$.
\end{itemize}
The next result shows that it is impossible to construct a $1$-dependent $4$-coloring of $\mathscr{S}^3_{(1,1,\infty)}$ and $\mathscr{S}^3_{(1,2,\infty)}$ which satisfies $(i)$-$(ii)$ above.
\begin{proposition}
\label{prop:main}
There is no $1$-dependent $4$-coloring of $\mathscr{S}^3_{(1,1,\infty)}$ (and hence $\mathscr{S}^3_{(1,2,\infty)}$) 
which satisfies (i)--(ii).
\end{proposition}

The proof of Proposition \ref{prop:main} is deferred to Section \ref{sc4}.
Note that Proposition \ref{prop:main} does not exclude the possibility of finding a $1$-dependent $4$ coloring of $\mathscr{S}^3_{(1,1,\infty)}$ or $\mathscr{S}^3_{(1,2,\infty)}$ with restriction to any copy of $\mathbb{Z}$ a symmetric $1$-dependent $4$-coloring of $\mathbb{Z}$. 
Proving or disproving this fact are both of interest. 
If one proves such a coloring is impossible, it gives the first example where the lower bound for the number of colors specified by the inverse critical point is not attained.
Otherwise, one can construct a $1$-dependent $4$-coloring whose restriction to ${\bf v}^1 v_0 {\bf v}^2$, ${\bf v}^1 v_0 {\bf v}^3$, ${\bf v}^2 v_0 {\bf v}^3$ is a symmetric coloring but different from that given by \eqref{eq:DCR1}.

So far we have not seen any example of $1$-dependent coloring of an infinite graph other than $\mathbb{Z}$, which satisfies $(i)$-$(ii)$.
Here we aim to construct the first such coloring of $\mathscr{S}^3_{(1,1,\infty)}$. 
Write $\tree{\bf x}{x_L \, x_0 \, x_R}$ for a generic $q$-coloring of $\mathscr{S}^3_{(1,1,n)}$, with ${\bf x}: = (x_1, \ldots, x_n) \in [q]^n$ and  $x_L, x_0, x_R \in [q]$.
Based on small cases computation, we arrive at the following construction which is in a similar flavor to \eqref{eq:DCR1} but in a more intricate way.

\begin{theorem}
\label{thm:main2}
Let $P(x_L x_0 x_R) = P^*(x_L x_0 x_R)$, and
for $n \ge 1$ and $\tree{\bf x}{x_L \, x_0 \, x_R}$ a proper coloring of $\mathscr{S}^3_{(1,1,n)}$, let
\begin{subequations}
\label{eq:DCR2}
\begin{align}
P\left(\tree{\bf x}{x_L \, x_0 \, x_R}\right) &= \frac{1}{C(n+2)D(n+1)}\left(\sum_{i = 1}^n C(2i + 1) P\Bigg(\tree{\widehat{\bf x}_i}{x_L \, x_0 \, x_R} \right) \notag \\
& \qquad \qquad \qquad \qquad \quad + C(1) P\left(\tree{\widehat{\bf x}_1}{x_L \, x_1 \, x_R} \right)\Bigg)  \quad  \mbox{ if } x_L \ne x_R, \label{eq:DCR21} \\
P\left(\tree{\bf x}{x_L \, x_0 \, x_R}\right) &=  P^*(x_L x_0 {\bf x}) - \sum_{x'_R \in [q] \setminus \{x_0, x_R\}} P\left(\tree{\bf x}{x_L \, x_0 \, x'_R}\right) \qquad \mbox{if } x_L = x_R, \label{eq:DCR22}
\end{align}
\end{subequations}
where $C(\cdot)$, $D(\cdot)$ are defined by \eqref{eq:CD}, and $P({\bf y})$ with ${\bf y} \in [q]^{k}$ for some $k$ is given by \eqref{eq:DCR1}.
Let $P\left(\tree{\bf x}{x_L \, x_0 \, x_R}\right) = 0$ if $\tree{\bf x}{x_L \, x_0 \, x_R}$ is not a proper coloring.
Then for each $q \ge 5$,
$P(\cdot)$ defines a $1$-dependent $q$-coloring of $\mathscr{S}^3_{(1,1,\infty)}$ which satisfies (i)--(ii).
\end{theorem}

The proof of Theorem \ref{thm:main2} is given in Section \ref{sc5}.
Here are a few examples of cylinder probabilities:
$P(121) = \frac{1}{q^2(q-1)}$, $P(123) = \frac{1}{q^2(q-2)}$,
\begin{align*}
& P\left(\tree{1}{1 \,2 \,1} \right) = \frac{1}{q^3(q-1)}, \quad P\left(\tree{3}{1 \,2 \,1} \right) = \frac{1}{q^3(q-2)}, \quad P\left(\tree{4}{1 \,2 \,3} \right) = \frac{1}{q^3(q-3)}, \\
& P\left(\tree{\,\,\,\,1}{1 \,2 \,3 \, 4} \right) = \frac{1}{q^3(q-1)(q-3)}, \quad P\left(\tree{\,\,\,\,1}{1 \,2 \,3 \, 1} \right) =P\left(\tree{\,\,\,\,1}{1 \,2 \,3 \, 2} \right) = \frac{1}{q^3(q-1)(q-2)}.
\end{align*}

Note that the construction \eqref{eq:DCR2} distinguishes whether $x_L = x_R$ or not.
The case $x_L \ne x_R$ is specified by a deletion-concatenation recursion \eqref{eq:DCR21} 
which is similar to \eqref{eq:DCR1}.
Additional difficulties come from the subtraction recursion \eqref{eq:DCR22}, where the nonnegativity is also not obvious.
Theorem \ref{thm:main2}, together with Proposition \ref{prop:main} give a whole picture of symmetric $1$-dependent colorings of $\mathscr{S}^3_{(1,1, \infty)}$ which satisfy the additional condition (ii). 
However, we do not know any symmetric coloring of $\mathscr{S}^3_{(1,1, \infty)}$ which does not satisfy (ii).
The situation is even more complicated for $1$-dependent colorings of $\mathscr{S}^3_{(m, n, \infty)}$ 
where at least one of $m, n$ is greater than $1$.
It remains unknown whether one can find a $1$-dependent $q$-coloring for any $q \ge 5$ of $\mathscr{S}^3_{(1,2,\infty)}$ which satisfies the conditions $(i)$-$(ii)$.
We hope this work can provide some insights into symmetric $1$-dependent colorings of general graphs, and trigger further research in this direction.

\medskip
{\bf Organization of the paper.}
In Section \ref{sc2}, we provide background on the coefficients $C(n)$ and $D(n)$ used in our construction of $1$-dependent coloring. We also explain where these formulas come from.
In Section \ref{sc3}, we give a proof of Theorem \ref{thm:main1} regarding the $1$-dependent hard-core processes.
Proposition \ref{prop:main} is proved in Section \ref{sc4}.
In Section \ref{sc5}, we prove Theorem \ref{thm:main2} which is the main result of this paper.
We conclude with a list of open problems in Section \ref{sc6}.

\section{Background on the Chebyshev coefficients}
\label{sc2}

This section provides background on the `mysterious' Chebyshev coefficients $C(n)$ and $D(n)$,
which appear in the constructions \eqref{eq:DCR1} and \eqref{eq:DCR2} of $1$-dependent colorings of $\mathbb{Z}$ and $\mathscr{S}^3_{(1,1,\infty)}$ respectively.
We also give hints of these constructions.
To start, we record some initial values of $C(n)$ and $D(n)$.
\begin{align*}
&C(0) = 1, \quad C(1) = \frac{\sqrt{q}}{2}, \quad C(2) = \frac{q-2}{2}, \quad C(3) = \frac{\sqrt{q}(q-3)}{2}, \quad C(4) = \frac{q^2 - 4q + 2}{2}; \\
&D(0) = 0, \quad D(1) = \sqrt{q}, \quad D(2) = q, \quad D(3) = \sqrt{q}(q-1), \quad D(4) = q(q-2).
\end{align*}
Note that $C(\cdot)$ is an even function, and $D(\cdot)$ is an odd function.
Moreover, $C(n)$ and $D(n)$ satisfy the recursion
\begin{equation}
\label{eq:recCD}
\mathscr{A}(n+2) - \sqrt{q} \, \mathscr{A}(n+1) + \mathscr{A}(n) = 0 \quad \mbox{for } \mathscr{A} = C, \, D.
\end{equation}

The following lemma collects a few useful identities of $C(n)$ and $D(n)$, which we will use later.
The results are read from \cite{HL15}, which are consequences of the recursion \eqref{eq:recCD}.
\begin{lemma}
For $j, \, k, \, \ell, \, m, \, n \in \mathbb{Z}$ and ${\bf x}$ a $q$-coloring of length $N > 0$, we have the following identities:
\begin{align}
& 2C(m)C(n) = C(m+n) + C(n-m), \label{eq:HL0} \\[3 pt]
& 2 C(m) D(n) = D(m+n) + D(n-m), \label{eq:HL1} \\[3 pt] 
&C(j+ k) D(k+ \ell) = C(k) D(j+k+\ell) - C(\ell) D(j), \label{eq:HL2} \\
& \sum_{i = 1}^N C(2i) P(\widehat{\bf x}_i) = C(N+1) D(N+1) P(\bf{x}). \label{eq:HL3}
\end{align}
where $P(\cdot)$ is the probability measure defined by \eqref{eq:DCR1}.
\end{lemma}

Here we explain why the constructions \eqref{eq:DCR1} and \eqref{eq:DCR2} are expected.
For a coloring ${\bf x}$ of $\mathbb{Z}$, assume that the recursion is of form
\begin{equation}
\label{eq:26}
P({\bf x}) = \sum_{i = 1}^n c_i P(\widehat{\bf x}_i).
\end{equation}

The deletion-concatenation recursion is similar in spirit to the deletion-contraction recursion, which is a powerful tool in enumerative combinatorics.
Given ${\bf x} = (x_1, \ldots, x_n)$ ${\bf y} = (y_1, \ldots, y_m)$ with $x_n \ne y_1$, 
the one-dependence condition implies that
$\sum_{a \ne x_n, y_1} P({\bf x} a {\bf y}) = P({\bf x})  P({\bf y})$.
By induction, we write
\begin{align}
\label{eq:27}
\sum_{a \ne x_n, y_1} P({\bf x} a {\bf y}) & = \sum_{i = 1}^n \sum_{a \ne x_n, y_1} c_i P(\widehat{\bf x}_ia{\bf y}) + (q-2) c_{n+1}P({\bf x y})  + \sum_{i = n+2}^{n+m+1} c_i P({\bf x}a\widehat{\bf y}_i)  \notag \\
& = \sum_{i = 1}^n c_i P(\widehat{\bf x}_i) P({\bf y}) + (-c_{n} + (q-2) c_n + c_{n+1}) P({\bf xy}) + \sum_{i = n+2}^{n+m+1} c_i P({\bf x}) P(\widehat{\bf y}_i).
\end{align}
Note that the first and the last term in \eqref{eq:27} can easily create the term $P({\bf x}) P({\bf y})$,
so the only annoying term is the middle one. Thus, we need to kill this term by sending its coefficient to zero.
For this reason, the coefficients $c_i$ in \eqref{eq:26} are forced to satisfy:
\begin{equation}
c_{i + 2} - (q-2) c_{i+1} +c_i = 0 \quad \mbox{for } 0 \le i \le n-1,
\end{equation}
with the convention $c_0 = c_{n+1}: = 0$.
Further by symmetry under reflection and small cases computation, the coefficients $c_i$ are uniquely determined which is given by \eqref{eq:DCR1}.
Similarly, for a coloring of $\tree{\bf x}{x_L \, x_0 \, x_R}$ of $\mathscr{S}^3_{1,1, \infty}$ with $x_L \ne x_R$, if assuming
\begin{equation*}
P \left( \tree{\bf x}{x_L \, x_0 \, x_R}\right) = \sum_{i = 1}^n c_i P\left(\tree{\widehat{\bf x}_i}{x_L \, x_0 \, x_R} \right)
+ c_0 P\left(\tree{\widehat{\bf x}_1}{x_L \, x_1 \, x_R} \right),
\end{equation*}
then the coefficients $c_i$ are expected to satisfy:
\begin{equation}
c_{i + 2} - (q-2) c_{i+1} +c_i = 0 \quad \mbox{for } 1 \le i \le n-1, \quad \mbox{and} \quad c_1 = (q-3)c_0,
\end{equation}
where the additional condition $c_1 = (q-3)c_0$ is due to the one-dependence condition at the distinguished vertex of degree $3$.
Again by small cases computation, the coefficients $c_i$ are uniquely determined, which is given as in \eqref{eq:DCR21}.

\section{$1$-dependent $q$-coloring of $\mathscr{S}^d$}
\label{sc3}
In this section, we prove Theorem \ref{thm:main1}.
In Section \ref{sc31}, we provide an elementary probabilistic proof of Theorem \ref{thm:main1}.
We also discuss how to derive the results from graph theory.
In Section \ref{sc32}, we give an alternative proof that there is no $1$-dependent $4$-coloring of the $3$-ray star graph.

\subsection{The critical point of the $d$-ray star graph}
\label{sc31}
For ${\bf x} = (x_v; \, v \in \mathscr{S}^d_{\bf n}) \in \{0,1\}^{\mathscr{S}^d_{\bf n}}$ with $d$ rays
${\bf x}_1:= (x_{11}, \ldots, x_{1n_1}), \ldots, {\bf x}_d:=(x_{d1}, \ldots, x_{dn_d})$ emanating from $x_0$, write
\begin{equation*}
Q({\bf x}) : = \mathbb{P}(J_{v} = x_v \mbox{ for all } v \in \mathscr{S}^d_{\bf n}).
\end{equation*}

If there exists a $1$-dependent hard-core process $J$ on $\mathscr{S}^d$ with marginals $\mathbb{P}(J_v = 1) = p$ for all $v$, 
then the collection $(Q({\bf x}); \, {\bf x} \in \{0,1\}^{\mathscr{S}^d_{\bf n}}, {\bf n} \in \mathbb{N}^d)$ is a nonnegative solution to the 1-dependence equations at all internal vertices, and the consistency equations at all $d$ leaves.
The following proposition shows that these cylinder probabilities are uniquely determined by the 1-dependence and consistency conditions.

\begin{proposition}
\label{keylem}
 Let $Q^*(\cdot)$ be the cylinder probabilities of the 1-dependent hard-core process on $\mathbb{Z}$ with marginals $p$.
For each ${\bf x} \in \{0,1\}^{\mathscr{S}^d_{\bf n}}$,
define
\begin{equation}
\label{Pformula}
Q({\bf x}) = \left\{ \begin{array}{ccl}
\prod_{i=1}^d  Q^*({\bf x}_i) & \mbox{if } & x_0 = 0 \mbox{ and } x_{i1} = 1 \mbox{ for some } i,
 \\[3 pt] 
 p  \prod_{i=1}^d Q^*(\widehat{{\bf x}}_{i1})& \mbox{if} & x_0 = 1 \mbox{ and } x_{i1} = 0 \mbox{ for all } i, \\[3 pt]
\prod_{i=1}^d Q^*({\bf x}_i) - p  \prod_{i=1}^d Q^*(\widehat{{\bf x}}_{i1}) & \mbox{if} & x_0 = 0 \mbox{ and } x_{i1} = 0 \mbox{ for all } i,
\end{array}\right.
\end{equation}
where $\widehat{{\bf x}}_{i1} : = (x_{i2}, \ldots, x_{in_i})$ is obtained by deleting the first entry of ${\bf x}_i$, 
with the convention $Q^*(\emptyset) = 1$.
If \begin{equation}
\label{nonnegative}
\prod_{i=1}^d Q^*({\bf x}_i) - p  \prod_{i=1}^d Q^*(\widehat{{\bf x}}_{i1}) \geq 0 \quad \mbox{for all } {\bf x}_i \in \{0\} \times \{0,1\}^{n_i} \mbox{ with } n_i \ge 0,
\end{equation}
then the collection $(Q({\bf x}); \, {\bf x} \in \{0,1\}^{\mathscr{S}^d_{\bf n}}, {\bf n} \in \mathbb{N}^d)$ given by \eqref{Pformula} defines
the unique $1$-dependent hard-core process $J$ on $\mathscr{S}^d$ with marginals $p$.
Conversely, if the condition \eqref{nonnegative} fails, then there is no 1-dependent hard-core process with marginals $p$.
\end{proposition}
\begin{proof}
Suppose that there is a $1$-dependent hard-core process on $\mathscr{S}^d$ with marginals $p$.
For ${\bf x} \in \{0,1\}^{\mathscr{S}^d_{\bf n}}$, there are three cases.

\smallskip
{\bf Case 1}: $x_0 = 0$ and one of the neighbors to $x_0$ is $1$. By the $1$-dependence condition at $x_0$, we get the desired result.

\smallskip
{\bf Case 2}: $x_0 = 1$ and all the neighbors to $x_0$ are $0$. By the $1$-dependence condition at $x_{11}$, we have
\begin{equation*}
Q({\bf x}) = Q^*(\widehat{\bf x}_{11}) Q^*(\widehat{\bf x}^1) = Q^*(\widehat{\bf x}_{11}) \cdot p  \prod_{i=2}^d Q^*(\widehat{{\bf x}}_{i1}),
\end{equation*}
where $\widehat{\bf x}^1$ is obtained by deleting the ray ${\bf x}_1$ from ${\bf x}$, and the second equality follows by applying the $1$-dependence condition successively at $x_{21}, \ldots, x_{d1}$.

\smallskip
{\bf Case 3}: $x_0 = 0$ and all the neighbors to $x_0$ are $0$. By the $1$-dependence condition at $x_0$ and the result in Case 2, we get the desired result.

It is easy to check that the collection $(Q({\bf x}); \, {\bf x} \in \{0,1\}^{\mathscr{S}^d_{\bf n}}, {\bf n} \in \mathbb{N}^d)$ given by \eqref{Pformula} satisfies the 1-dependence and consistency conditions. We leave the full detail to readers. 
The non-negativity condition is guaranteed by \eqref{nonnegative}. We conclude by applying Kolmogorov's extension theorem.
\end{proof}

\quad As a consequence of Proposition \ref{keylem}, we have
\begin{equation}
\label{criticalrep1}
p_h(\mathscr{S}^d) = \sup \left\{p; \, p \le \prod_{i=1}^d \frac{Q^*({\bf x}_i)}{Q^*(\widehat{{\bf x}}_{i1})}
\mbox{ for all } {\bf x}_i \in \{0\} \times \{0,1\}^{n_i} \mbox{ with }  n_i < \infty \right\}.
\end{equation}

For ${\bf x} \in \{0\} \times \{0,1\}^n \setminus \{0\}^{n+1}$, let
\begin{equation*}
k({\bf x}): = \inf\{k; \, x_k = 1\} - 1
\end{equation*}
be the number of $0$'s before the first $1$ in ${\bf x}$. 
Write ${\bf x} = 0_{k({\bf x})} {\bf S}$ with ${\bf S} \in \{1\} \times \{0,1\}^{n-k({\bf x})}$. 
By the $1$-dependence condition at the $k^{th}$ position, we have
\begin{equation*}
Q^*({\bf x}) = Q^*(0_{k({\bf x})-1}) Q^*({\bf S}),
\end{equation*}
with the convention $Q^*(0_0) : =1$.
So for ${\bf x} \in \{0\} \times \{0,1\}^n \setminus \{0\}^{n+1}$, 
\begin{equation}
\label{zero}
\frac{Q^*({\bf x})}{Q^*(\widehat{{\bf x}}_{1})} = \left\{ \begin{array}{ccl}
1 & \mbox{if } & k({\bf x}) = 1
 \\ 
Q^*(0_{k({\bf x})-1})/ Q^*(0_{k({\bf x})-2}) & \mbox{if} & k({\bf x}) \ge 2.
\end{array}\right.
\end{equation}

It follows from \eqref{criticalrep1} and \eqref{zero} that
\begin{equation}
\label{criticalrep2}
p_h(\mathscr{S}^d_{\bf n}) =  \sup \left\{p ; \,  p \le \prod_{i = 1}^d \inf_{k \le n_i} \frac{Q^*(0_{k})}{Q^*(0_{k-1})} \right\}.
\end{equation}

Thus the value of $p_h(\mathscr{S}^d_{\bf n})$ is entirely determined by the sequence $(Q^*(0_k); \, k \ge 0)$.
The following lemma gives an explicit recursion of the sequence $(Q^*(0_k); \, k \ge 0)$.
\begin{lemma}
For each $k \ge 2$,
\begin{equation}
\label{recurrence}
Q^*(0_k) = Q^*(0_{k-1}) - p Q^*(0_{k-2}).
\end{equation}
In particular,
\begin{equation*}
Q^*(0_1) = 1-p, \quad Q^*(0_2) = 1 - 2p, \quad Q^*(0_3) = 1-3p+p^2, \quad Q^*(0_4) = 1 - 4p + 3p^2 \ldots
\end{equation*}
\end{lemma}
\begin{proof}
By the consistency condition of $0_k$, we have
\begin{equation*}
Q^*(0_k)  = Q^*(0_{k-1}) - Q^*(0_{k-1}1).
\end{equation*}
Further by the 1-dependence condition of $0_{k-1}1$ at the $k-1^{th}$ position, we get
\begin{equation*}
Q^*(0_{k-1}1) = Q^*(0_{k-2}) Q^*(1),
\end{equation*}
which yields \eqref{recurrence}.
\end{proof}

For $k \ge 1$, let $a_k(p) : = Q^*(0_k)/Q^*(0_{k-1})$.
The recursion \eqref{recurrence} implies \eqref{recsub}, which proves the formula \eqref{eq:phSS}.
Specializing to the $d$-ray star graph $\mathscr{S}^d$, we get
\begin{equation}
\label{criticalrepe3}
p_h(\mathscr{S}^d) =  \sup \left\{p ; \,  p \le  \left(\inf_{k \ge 1} a_k(p)\right)^d \right\}.
\end{equation}
It follows by standard analysis that for $p \le 1/4$, the sequence $a_k$ decreases to the limit $(1+ \sqrt{1-4p})/2$.
Combining this with \eqref{criticalrepe3} yields the equation \eqref{eqph}.

\begin{remark}
As mentioned in the introduction, Theorem \ref{thm:main1} (and Proposition \ref{keylem}) can be derived from known results in graph theory. 
For $G$ a finite graph, consider the (modified) independence polynomial
\begin{equation*}
Z_G(x) = \sum_{k = 0}^n (-1)^k i_k(G) x^k,
\end{equation*}
where $i_k(G)$ is the number of independent sets of size $k$ with the convention $i_0(G) = 1$.
Then $p_h(G)$ is the smallest real zero of $Z_G(x)$.
For an infinite graph, it suffices to take the infimum of the smallest real zeros of the finite subgraphs of $G$.
It is well known that $Z_{G}(x)$ satisfies the recursion
\begin{equation*}
Z_G(x) = Z_{G \setminus v}(x) - x Z_{G \setminus N[v]}(x),
\end{equation*}
where $N[v] =\{v\} \cup N(v)$, that is, vertex $v$ together with its set of neighbors $N(v)$.
Specializing this recursion to $\mathscr{S}^d_{\bf n}$, we get
\begin{equation}
\label{eq:38}
Z_{\mathscr{S}^d_{\bf n}}(x) = \prod_{i = 1}^d Z_{P_{n_i}}(x) - x \prod_{i = 1}^d Z_{P_{n_i-1}}(x),
\end{equation}
where $P_k$ is the path on $k$ vertices. 
The formula \eqref{eq:38} is essentially a reformulation of Proposition \ref{keylem}, from which follows Theorem \ref{thm:main1}.
\end{remark}

\subsection{No symmetric $1$-dependent $4$-coloring of the $3$-ray star graph}
\label{sc32}
Here we give a direct proof that there is no $1$-dependent $4$-coloring of the the $3$-ray star graph
which satisfies $(i)$ and 
\begin{enumerate}
\item[($ii'$)]
the restriction to any path is given by a proper coloring which is invariant under permutations of the colors, 
under translation and reflection.
\end{enumerate}
The distribution of such a proper coloring of $\mathbb{Z}$ is denoted as $\widetilde{P}^*$.

Suppose by contradiction that such a coloring exists. 
Let $\bf T$ and $\bf T'$ be two colorings of $\mathcal{S}^3_{(2,2,4)}$ as follows:
\begin{equation}
{\bf T}:= \, 
\begin{matrix} 
   &    &    &    & 1 &    &  \\
   &    &    &    & 2 &    &  \\
1 & 2 & 3 & 4 & 1 & 3 & 2 
\end{matrix} 
\quad \mbox{and} \quad 
{\bf T}':= \,
\begin{matrix} 
   &    &    &    & 1 &    &  \\
   &    &    &    & 2 &    &  \\
1 & 2 & 3 & 2 & 1 & 3 & 2 
\end{matrix} 
\end{equation}

\quad The $1$-dependence condition of $\bf T$ at the distinguished vertex of degree $3$ implies that
\begin{equation*}
P({\bf T}) =  \widetilde{P}^*(1234)  \widetilde{P}^*(12)  \widetilde{P}^*(32)  =  P^*(1234) ( \widetilde{P}^*(12))^2,
\end{equation*}
while the $1$-dependence condition of $\bf T$ at the left neighbor to the distinguished vertex implies that
\begin{equation*}
P({\bf T}) + P({\bf T'}) =  \widetilde{P}^*(123)   \widetilde{P}^*(12132).
\end{equation*}
Therefore,
\begin{equation}
\label{cond2}
 \widetilde{P}^*(123) \widetilde{P}^*(12132) \ge \widetilde{P}^*(1234) (\widetilde{P}^*(12))^2.
\end{equation}
Recall from \cite{HL15} that for a $1$-dependent $4$-coloring of the integers which is invariant under permutations of the colors, under translation and reflection, 
the cylinder probabilities of length $k \le 4$ are given by
\begin{equation}
\label{cyl4}
\begin{split}
&  \widetilde{P}^*(1) = \frac{1}{4}, \quad  \widetilde{P}^*(12) = \frac{1}{12}, \quad  \widetilde{P}^*(121) = \frac{1}{48}, \quad \widetilde{P}^*(123) = \frac{1}{32}, \\
&  \widetilde{P}^*(1212) = \frac{\alpha}{48}, \quad  \widetilde{P}^*(1213) = \frac{1-\alpha}{96}, \quad  \widetilde{P}^*(1231) = \frac{1}{96}, \quad \widetilde{P}^*(1234) = \frac{1+\alpha}{96}.
\end{split}
\end{equation}
So the cylinder probabilities of length $k \le 3$ are uniquely determined, while those of length $k = 4$ are given by a one parameter family indexed by $\alpha$. 

Now we consider the cylinder probabilities of length $k = 5$. 
There are $10$ equivalent classes under permutations of colors and under reflection: $12134$, $12314$, $12324$, $12341$, $12131$, $12123$, $12132$, $12312$, $12321$ and $12121$. 
Solving the $1$-dependence and consistency conditions for these cylinder probabilities gives the following result.

\begin{lemma}
For a $1$-dependent $4$-coloring of $\mathbb{Z}$ which is invariant in law under permutations of the colors, under translation and reflection, with the cylinder probabilities of length $k \le 4$ given by \eqref{cyl4}, the cylinder probabilities of length $k = 5$ are determined by
\begin{align}
& \widetilde{P}^*(12134) = \frac{1}{288}, \,\, \widetilde{P}^*(12314) = \frac{5}{1152}, \,\,  \widetilde{P}^*(12324) = \frac{1}{288}, \,\, \widetilde{P}^*(12341) = \frac{1 + 4 \alpha}{384}, \notag \\
& \widetilde{P}^*(12131) = \frac{5 - 12 \alpha}{1152}, \,\, \widetilde{P}^*(12123) = \frac{1}{576}, \,\,  \widetilde{P}^*(12132) = \frac{1}{384}, \,\,  \widetilde{P}^*(12312) = \frac{1}{288}, \\
& \widetilde{P}^*(12321) = \frac{1-2 \alpha}{192}, \,\,  \widetilde{P}^*(12121) = \frac{6 \alpha - 1}{288}. \notag
\end{align}
\end{lemma}
So the nonnegative condition for cylinder probabilities requires $\frac{1}{6} \le \alpha \le \frac{5}{12}$, which contradicts $\alpha \le 1/8$ given by \eqref{cond2}.

\begin{remark}
As indicated in \cite{HL15}, the construction of a $1$-dependent $q$-coloring of $\mathbb{Z}$ requires finding a nonnegative solution to an infinite set of nonlinear equations. The unknowns are the cylinder probabilities, which are assumed to be symmetric in the colors, translation invariant and invariant under reflection. For each $q \ge 4$ and $k \ge 1$, let
\begin{equation*}
L_{q, k}: = \# \mbox{ unknowns corresponding to the cylinder probabilities of length } k.
\end{equation*}
So $L_{q,n}$ is the number of the equivalent classes of $q$-coloring of $[k]$ under permutations of the colors and under reflection. 
The number of equations for these $L_{q,k}$ unknowns is then $(k-1)L_{q,k}$, among which $(k-2) L_{q,k}$ are $1$-dependence conditions, and $L_{q,k}$ are consistency conditions.
For $q = 4$, the sequence $(L_{4,k}; \, k \ge 1)$ is given by the OEIS A001998 \cite{OEIS}. 
In particular,
\begin{equation}
L_{4, k} = \frac{1+3^{k-1}+3^{(k-2)/2} (2 + \sqrt{3} - (-1)^{k-1} (2 - \sqrt{3}))}{4} \quad \mbox{for } k \ge 1.
\end{equation}
To illustrate, $L_{4,1} = 1$, $L_{4,2} = 1$, $L_{4,3} = 2$, $L_{4,4} = 4$, $L_{4,5} = 10$, $L_{4,6} = 25$, $L_{4,7} = 70$, $L_{4,8} = 196$, $L_{4,9}= 574$, $L_{4,10} = 1681 \ldots$
By further solving the equations for cylinder probabilities of length $k = 6$, we get a two parameter family of $(\alpha, \beta)$ for $L_{4, 6} = 25$ unknowns.  
But it is not obvious there is a unique nonnegative solution to these equations unless the non-negativity conditions force $\alpha = 1/5$ as $k \rightarrow \infty$.
\end{remark}

\section{No symmetric $1$-dependent $4$-coloring of $\mathscr{S}^3_{(1,1,\infty)}$ under the condition $(ii)$}
\label{sc4}

In this section, we deal with Proposition \ref{prop:main}.
The idea is to show that the condition $(ii)$ implies a `negative' probability for some $4$-coloring of $\mathscr{S}^3_{(1,1,\infty)}$.
Suppose by contradiction that such a coloring exists.
Let ${\bf T}_{n}$ be the coloring of $\mathscr{S}^3_{(1,1,n)}$ defined by
\begin{equation}
\label{eq:T12}
{\bf T}_{n}: = \,
\begin{matrix} 
   & 4 &   \\
3 & 1 & \smash[b]{\block{n}}  
\end{matrix} 
\end{equation}

\bigskip
Also define ${\bf R}_n$ and ${\bf R}'_n$ as:
\begin{equation}
\label{eq:Rn}
{\bf R}_n: = \,  \smash[b]{\blockb{n}} \quad \mbox{and} \quad {\bf R}'_n: = \,3\, \smash[b]{\blockb{n-1}}
\end{equation}

\bigskip
We need the following lemma which computes the probabilities of ${\bf R}_n$ and ${\bf R}'_n$.

\begin{lemma}
\label{lem:probR}
Let $P^{*}(\cdot)$ be the family of probabilities defined by \eqref{eq:DCR1}, and ${\bf R}_n$, ${\bf R}'_n$ be defined by \eqref{eq:Rn}.
Then we have
\begin{align}
& P^{*}({\bf R}_n) = \frac{1}{2(n+1) !},  \quad \mbox{for } n \ge 1,  \label{eq:43}\\
& P^{*}({\bf R}'_n) = \frac{1}{4(n+1)(n-1)!}, \quad  \mbox{for } n \ge 2. \label{eq:44}
\end{align}
\end{lemma}

\begin{proof}
It is easily seen from the recursion \eqref{eq:DCR1} that
\begin{equation*}
P^{*}({\bf R}_n) = \frac{1}{2(n+1)} (P^{*}(\widehat{\bf R}_{n,1}) + P^{*}(\widehat{\bf R}_{n,n})) = \frac{P^{*}({\bf R}_{n-1})}{n+1},
\end{equation*}
where the second equality is due to the symmetry of the colors.
This leads to the formula \eqref{eq:43}.
The formula \eqref{eq:44} follows from the fact that $P^{*}({\bf R}'_n) = \frac{1}{2}(P^{*}({\bf R}_{n-1}) - P^{*}({\bf R}_{n}))$.
\end{proof}

Now the $1$-dependence condition at the distinguished vertex implies that 
\begin{equation}
\label{eq:45}
P({\bf T}_{n}) = \frac{1}{16}  P^{*}({\bf R}_n).
\end{equation}

On the other hand, the probability in \eqref{eq:45} are no larger than
$P^{*}(4 \, \smash[b]{\blockb{n+1}}) =  P^{*}({\bf R}'_{n+2})$.
By Lemma \ref{lem:probR}, this is equivalent to
\begin{equation}
\label{eq:contradb}
\frac{1}{32 (n+1) !} \le \frac{1}{4(n+3) (n+1) !},
\end{equation}
which fails for $n > 5$. 
This yields the contradiction.

\section{Construction of a $1$-dependent $q$-coloring of $\mathscr{S}^3_{(1,1,\infty)}$}
\label{sc5}

This section is devoted to the proof of Theorem \ref{thm:main2}.
In Section \ref{sc51} we check the consistency and one-dependence conditions.
The nonnegativity conditions are proved in Section \ref{sc52}.

\subsection{Consistency and one-dependence conditions}
\label{sc51}
We check that $P(\cdot)$ defined by \eqref{eq:DCR2} satisfies the consistency and $1$-dependence conditions. 
This is a tedious case-by-case induction. 
We start with the case $x_L \ne x_R$. 
By symmetry of the colors, it suffices to consider $P\left(\tree{{\bf x}}{1 \, 2 \, 3}\right)$, with ${\bf x} \in [q]^k$ for some $k$.
The results are stated in the following proposition.

\begin{proposition}
\label{prop:51}
Let $P(\cdot)$ be defined as in Theorem \ref{thm:main2}, 
and let $a \in [q]$, ${\bf x} = (x_1, \ldots, x_n) \in [q]^n$ and ${\bf y} = (y_1, \ldots, y_m) \in [q]^m$.
We have
\begin{enumerate}
\item[(i)]
$\sum_{a \ne 2} P \left(\tree{a}{1 \, 2 \, 3} \right) = P(123)$.
\item[(ii)]
$\sum_{a \ne x_n} P\left(\btree{a}{\bf x}{1 \, 2 \, 3}\right) = P \left(\tree{\bf x}{1 \, 2 \, 3} \right)$.
\item[(iii)]
$\sum_{a \ne x_n, y_1} P\left(\ttree{\bf y}{a}{\bf x}{1 \, 2 \, 3} \right) = P^*({\bf y}) P\left(\tree{\bf x}{1 \, 2 \, 3} \right)$.
\item[($iv$)]
$\sum_{a \ne 2, x_1} P\left(\btree{\bf x}{a}{1 \, 2 \, 3 } \right) = P^*({\bf x}) P(123)$.
\item[$(v)$]
$\sum_{a \ne 1,2,x_1} P \left(\tree{\bf x}{1 \, a \, 2} \right) = P^*({\bf x}) P^*(1) P^*(2)$.
\end{enumerate}
\end{proposition}

The consistency and one-dependence of $P\left( \tree{\bf x}{1 \, 2 \, 1}\right)$, as well as the consistency of $x_L$ and $x_R$ can be deduced easily from the subtraction construction \eqref{eq:DCR22}.
We summarize the corresponding results in the following corollary, and leave the details to the readers. 

\begin{corollary}
Let $P(\cdot)$ be defined as in Theorem \ref{thm:main2}, 
and let $a \in [q]$, ${\bf x} = (x_1, \ldots, x_n) \in [q]^n$ and ${\bf y} = (y_1, \ldots, y_m) \in [q]^m$.
We have
\begin{enumerate}
\item[(i)]
$\sum_{a \ne 2} P \left(\tree{a}{1 \, 2 \, 1} \right) = P(121)$.
\item[(ii)]
$\sum_{a \ne x_n} P\left(\btree{a}{\bf x}{1 \, 2 \, 1}\right) = P \left(\tree{\bf x}{1 \, 2 \, 1} \right)$.
\item[(iii)]
$\sum_{a \ne x_n, y_1} P\left(\ttree{\bf y}{a}{\bf x}{1 \, 2 \, 1} \right) = P^*({\bf y}) P\left(\tree{\bf x}{1 \, 2 \, 1} \right)$.
\item[($iv$)]
$\sum_{a \ne 2, x_1} P\left(\btree{\bf x}{a}{1 \, 2 \, 1 } \right) = P^*({\bf x}) P(121)$.
\item[$(v)$]
$\sum_{a \ne 1,x_1} P \left(\tree{\bf x}{1 \, a \, 1} \right) = P^*({\bf x}) P^*(1) P^*(1)$.
\item[$(vi)$]
$\sum_{a \ne x_0} P \left(\tree{\bf x}{x_L \, x_0 \, a} \right) = P^*(x_L x_0 {\bf x})$ \, and \,
$\sum_{a \ne x_0} P \left(\tree{\bf x}{a \, x_0 \, x_R} \right) = P^*(x_R x_0 {\bf x})$.
\end{enumerate}
\end{corollary}

\begin{proof}[Proof of Proposition \ref{prop:51}]
($i$) By the construction \eqref{eq:DCR21}, 
\begin{equation*}
P \left(\tree{a}{1 \, 2 \, 3} \right) = \frac{1}{C(3) D(2)} \left(C(3) P(123) + C(1) P(1a3) \right) \quad \mbox{for } a \ne 1,2,3,
\end{equation*}
\begin{equation*}
P \left(\tree{a}{1 \, 2 \, 3} \right) = \frac{1}{D(2)} P(123)  \quad \mbox{for } a = 1, 3.
\end{equation*}
As a result,
\begin{align*}
\sum_{a \ne 2} P \left(\tree{a}{1 \, 2 \, 3} \right) &= \frac{q-1}{D(2)} P(123) + \frac{C(1)}{C(3)D(2)} \sum_{a \ne 1,2,3} P(1a3) \\
& = \frac{q-1}{q} P(123) + \frac{(q-3) \sqrt{q}/2}{q^{3/2}(q-3)/2} P(123) = P(123).
\end{align*}

($ii$)
By the construction \eqref{eq:DCR21}, 
\begin{multline*}
P\left(\btree{a}{\bf x}{1 \, 2 \, 3}\right) = \frac{1}{C(n+3)D(n+2)} \Bigg( C(2n+3) P\left(\tree{\bf x}{1 \, 2 \, 3} \right) \\
+ \sum_{i=1}^n C(2i + 1) P\left(\btree{a \,\,}{\widehat{\bf x}_i}{1 \, 2 \, 3} \right) + C(1) P \left( \btree{a \,\,\,}{\widehat{\bf x}_1}{1 \, x_1 \, 3}\right)  \Bigg),
\end{multline*}
which implies that 
\begin{align}
\label{eq:51}
& \sum_{a \ne x_n} P\left(\btree{a}{\bf x}{1 \, 2 \, 3}\right) \notag  \\
& = \frac{1}{C(n+3)D(n+2)}\Bigg[ (q-1) C(2n+3) P\left(\tree{\bf x}{1 \, 2 \, 3} \right) + C(2n+1) \sum_{a \ne x_n} P\left(\btree{a \,\,\,}{\widehat{\bf x}_n}{1 \, 2 \, 3} \right) \notag  \\
& \qquad \qquad \qquad  \qquad \qquad + \sum_{i = 1}^{n-1} C(2i + 1) \sum_{a \ne x_n} P\left(\btree{a \,\,\,}{\widehat{\bf x}_i}{1 \, 2 \, 3} \right) + C(1) \sum_{a \ne x_n} P \left( \btree{a \,\,\,}{\widehat{\bf x}_1}{1 \, x_1 \, 3}\right) \Bigg] \notag  \\
& = \frac{1}{C(n+3)D(n+2)}\Bigg[ (q-1) C(2n+3) P\left(\tree{\bf x}{1 \, 2 \, 3} \right) + C(2n+1) \left(P \left(\tree{\widehat{\bf x}_n}{1 \, 2 \, 3}\right) - P \left( \tree{\bf x}{1 \, 2 \, 3}\right) \right) \notag  \\
& \qquad \qquad \qquad  \qquad \qquad \qquad \qquad \quad + \sum_{i = 1}^{n-1} C(2i + 1) P\left(\tree{\widehat{\bf x}_i}{1 \, 2 \, 3} \right) + C(1)  P \left( \tree{\widehat{\bf x}_1}{1 \, x_1 \, 3}\right) \Bigg] \notag  \\
& = \frac{1}{C(n+3)D(n+2)}\Bigg[(q-1) C(2n+3) P\left(\tree{\bf x}{1 \, 2 \, 3} \right) - C(2n+1)P \left( \tree{\bf x}{1 \, 2 \, 3}\right) \notag  \\
& \qquad \qquad \qquad  \qquad \qquad \qquad \qquad \quad +  \underbrace{\sum_{i = 1}^{n} C(2i + 1) P\left(\tree{\widehat{\bf x}_i}{1 \, 2 \, 3} \right) + C(1)  P \left( \tree{\widehat{\bf x}_1}{1 \, x_1 \, 3}\right)}_{= C(n+2) D(n+1) P\left(\tree{\bf x}{1 \, 2 \, 3} \right) \, \mbox{\tiny by induction}} \Bigg] \notag  \\
& = \frac{(q-1) C(2n+3) - C(2n+1) + C(n+2) D(n+1)}{C(n+3) D(n+2)}P\left(\tree{\bf x}{1 \, 2 \, 3} \right).
\end{align}

Now using \eqref{eq:HL2} with $j = 2$, $k = -2n - 3$ and $\ell = 2n + 4$, we have
$\sqrt{q} C(2n+1) = \sqrt{q}(q-1) C(2n+3) - q C(2n+4)$, which implies that
\begin{equation*}
(q-1) C(2n+3) - C(2n+1) = C(2n+4) D(1).
\end{equation*}
Further using \eqref{eq:HL2} with $j = -n-1$, $k = 2n + 4$ and $\ell = -n-2$, we get
\begin{align}
\label{eq:52}
C(n+3) D(n+2) & = C(2n+4) D(1) + C(n+2) D(n+1) \notag  \\
& = (q-1) C(2n+3) - C(2n+1) + C(n+2) D(n+1),
\end{align}
Combining \eqref{eq:51} and \eqref{eq:52} leads to the desired result.

\smallskip
($iii$)
There are two cases.

\smallskip
{\em Case 1:} $x_n = y_1$. We have:
\begin{multline*}
P\left(\ttree{\bf y}{a}{\bf x}{1 \, 2 \, 3} \right) = \frac{1}{C(m+n+3)D(m+n+2)} \Bigg( \sum_{j = 1}^m C(2n+3+2j) P\left(\ttree{\widehat{\bf y}_j}{a}{\bf x}{1 \, 2 \, 3} \right) \\
+ \sum_{i=1}^n C(2i + 1) P\left(\ttree{\bf y \,}{a \,}{\widehat{\bf x}_i}{1 \, 2 \, 3} \right) + C(1) P\left(\ttree{\bf y \,}{a \,}{\widehat{\bf x}_1}{1 \, x_1 \, 3} \right)  \Bigg),
\end{multline*}
which gives
\begin{align}
\label{eq:53}
\sum_{a \ne x_n, y_1} P\left(\ttree{\bf y}{a}{\bf x}{1 \, 2 \, 3} \right) &= \frac{1}{C(m+n+3)D(m+n+2)} \Bigg[ \sum_{j = 1}^m C(2n+3+2j) \sum_{a \ne x_n, y_1} P\left(\ttree{\widehat{\bf y}_j}{a}{\bf x}{1 \, 2 \, 3} \right) \notag  \\
& \qquad \qquad + \sum_{i=1}^n C(2i + 1) \sum_{a \ne x_n, y_1} P\left(\ttree{\bf y \,}{a \,}{\widehat{\bf x}_i}{1 \, 2 \, 3} \right) + C(1) \sum_{a \ne x_n, y_1} P\left(\ttree{\bf y \,}{a \,}{\widehat{\bf x}_1}{1 \, x_1 \, 3} \right)  \Bigg] \notag  \\
& = \frac{1}{C(m+n+3)D(m+n+2)} \Bigg[ \sum_{j = 1}^m C(2n+3+2j) P^*(\widehat{\bf y}_j) P\left(\tree{\bf x}{1 \, 2 \, 3} \right) \notag  \\
& \qquad \qquad + P({\bf y}) \underbrace{\left(\sum_{i = 1}^n C(2i+1) P\left(\tree{\widehat{\bf x}_i}{1 \, 2 \, 3} \right) + C(1) P\left(\tree{\widehat{\bf x}_1}{1 \, x_1 \, 3} \right) \right)}_{= C(n+2) D(n+1) P\left(\tree{\bf x}{1 \, 2 \, 3} \right) \, \mbox{\tiny by induction}} \Bigg].
\end{align}

Now using \eqref{eq:HL2} with $j = 2n+m+4$, $k = -m + 2j -1$ and $\ell = -2j$, we get
\begin{equation*}
-C(2n+2j + 3) D(m+1) = C(m - 2j + 1) D(2n+3) - C(2j) D(2n + m +4),
\end{equation*}
which implies that
\begin{align}
\label{eq:54}
\sum_{j = 1}^m C(2n+3+2j) P^*(\widehat{\bf y}_j) & = \frac{1}{D(m+1)} \Bigg[-D(2n+3) \underbrace{\sum_{j = 1}^m C(m - 2j+1) P^*(\widehat{\bf y}_j)}_{= D(m+1) P^*({\bf y}) \, \mbox{\tiny by induction}} \notag  \\
& \qquad \qquad \qquad \quad + D(2n+m+4) \underbrace{\sum_{j = 1}^m C(2j) P^*(\widehat{\bf y}_j)}_{= D(m+1)C(m+1) \, \mbox{\tiny by \eqref{eq:HL3}}}  \Bigg] \notag  \\
& = \left(C(m+1) D(2n+m+4) - D(m+1) \right) P^*({\bf y}).
\end{align}

By \eqref{eq:HL1}, we have $C(m+1) D(2n+m+4) = \frac{1}{2} (D(2n+2m + 5) + D(m+1))$ so
\begin{align}
\label{eq:55}
C(m+1) D(2n+m+4) - D(m+1) &= \frac{1}{2} (D(2n+2m + 5) - D(m+1)) \notag  \\
& = C(2n+m+4) D(m+1).
\end{align}
Therefore,
$\sum_{a \ne x_n, y_1} P\left(\ttree{\bf y}{a}{\bf x}{1 \, 2 \, 3} \right) = \frac{C(2n+m+4) D(m+1) + C(n+2) D(n+1)}{C(m+n+3) D(m+n+2)} P^*({\bf y}) P\left(\tree{\bf x}{1 \, 2 \, 3} \right)$
by combining \eqref{eq:53}, \eqref{eq:54} and \eqref{eq:55}.
Again using \eqref{eq:HL2} with $k = m+n+3$, $j = n+1$ and $\ell = -n -2$, we get the desired result.

\smallskip
{\em Case 2:} $x_n \ne y_1$.
A similar argument as in Case 1 shows that
\begin{align}
\label{eq:56}
& \sum_{a \ne x_n, y_1} P\left(\ttree{\bf y}{a}{\bf x}{1 \, 2 \, 3} \right) \notag \\
&= \frac{1}{C(m+n+3)D(m+n+2)} \Bigg[ \sum_{j = 2}^m C(2n+3+2j) \sum_{a \ne x_n, y_1} P\left(\ttree{\widehat{\bf y}_j}{a}{\bf x}{1 \, 2 \, 3} \right) \notag \\
& \qquad \qquad \qquad \qquad + C(2n+5) \sum_{a \ne x_n, y_1} P\left(\ttree{\widehat{\bf y}_1}{a}{\bf x}{1 \, 2 \, 3} \right) + (q-2) C(2n+3) P\left(\btree{\bf y}{\bf x}{1 \, 2 \, 3} \right) \notag \\
& \qquad \qquad \qquad \qquad  + C(2n+1) \sum_{a \ne x_n, y_1} P\left(\ttree{\bf y \,}{a \,}{\widehat{\bf x}_n}{1 \, 2 \, 3} \right)
+ \sum_{i=1}^{n-1} C(2i + 1) \sum_{a \ne x_n, y_1} P\left(\ttree{\bf y \,}{a \,}{\widehat{\bf x}_i}{1 \, 2 \, 3} \right) \notag \\
& \qquad \qquad \qquad \qquad + C(1) \sum_{a \ne x_n, y_1} P\left(\ttree{\bf y \,}{a \,}{\widehat{\bf x}_1}{1 \, x_1 \, 3} \right)  \Bigg] \notag \\
& = \frac{1}{C(m+n+3)D(m+n+2)} \Bigg[\sum_{j = 2}^m C(2n+3+2j) P^*(\widehat{\bf y}_j) P\left(\tree{\bf x}{1 \, 2 \, 3} \right) \notag \\
& \qquad + C(2n+5) \left(P^*(\widehat{\bf y}_1) P\left(\tree{\bf x}{1 \, 2 \, 3} \right) - P\left( \btree{\bf y}{\bf x}{1 \, 2 \, 3}  \right)      \right) + (q-2) C(2n+3) P\left(\btree{\bf y}{\bf x}{1 \, 2 \, 3} \right) \notag \\
& \qquad + C(2n+1) \left( P^*({\bf y}) P\left(\tree{\widehat{\bf x}_n}{1 \, 2 \, 3} \right) - P\left( \btree{\bf y}{\bf x}{1 \, 2 \, 3} \right)\right) + \sum_{i = 1}^{n-1} C(2i+1) P^*({\bf y}) P\left(\tree{\widehat{\bf x}_i}{1 \, 2 \, 3} \right) \notag \\
& \qquad + C(1) P\left(\tree{\widehat{\bf x}_1}{1 \, x_1 \, 3} \right) P^*({\bf y}) \Bigg] \notag \\
& = \frac{1}{C(m+n+3)D(m+n+2)} \Bigg[  \sum_{j = 1}^m C(2n+3+2j) \sum_{a \ne x_n, y_1} P\left(\ttree{\widehat{\bf y}_j}{a}{\bf x}{1 \, 2 \, 3} \right) \notag \\
& \quad \qquad  \qquad \qquad \qquad+ ((q-2) C(2n+3) - C(2n+5) -C(2n+1)) P\left(\btree{\bf y}{\bf x}{1 \, 2 \, 3} \right) \notag \\
&  \quad \qquad  \qquad \qquad \qquad +\left(\sum_{i = 1}^n C(2i+1) P\left(\tree{\widehat{\bf x}_i}{1 \, 2 \, 3} \right) + C(1) P\left(\tree{\widehat{\bf x}_1}{1 \, x_1 \, 3} \right) \right) P^*({\bf y}) \Bigg].
\end{align}

Now using \eqref{eq:HL0} with $m = 2$, we get $(q-2)C(2n+3) - C(2n+5) - C(2n+1) = 0$.
So the equation \eqref{eq:56} reduces to \eqref{eq:53}, and the rest follows the same as in Case 1.

\smallskip
($iv$) 
There are three cases.

\smallskip
{\em Case 1:} $x_1 = 1$ or $3$. Assume w.l.o.g. that $x_1 = 1$.  
Then we have:
\begin{align}
\label{eq:57}
& \sum_{a \ne 1,2} P \left(\btree{\bf x}{a}{1 \,2 \, 3} \right) \notag \\
&= \frac{1}{C(n+3) D(n+2)} \Bigg[\sum_{i=2}^n C(2i + 3)\sum_{a \ne 1,2} P\left(\btree{\widehat{\bf x}_i}{a}{1 \, 2 \, 3} \right) + C(5) \sum_{a \ne 1,2} P \left(\btree{\widehat{\bf x}_1}{a}{1 \,2 \, 3} \right) \notag \\
& \qquad \qquad \qquad \qquad \qquad \qquad + (q-2) C(3) P\left(\tree{\bf x}{1 \, 2 \, 3} \right) + C(1) \sum_{a \ne 1,2,3} P \left(\tree{\bf x}{1 \,a \, 3} \right) \Bigg] \notag \\
& = \frac{1}{C(n+3) D(n+2)} \Bigg[\sum_{i=2}^n C(2i + 3) P^*(\widehat{\bf x}_i) P(123) + C(5) \left(P^*(\widehat{\bf x}_1) P(123) - P\left(\tree{\bf x}{1 \, 2 \, 3} \right) \right) \notag  \\
&\qquad \qquad \qquad \qquad \qquad \qquad + (q-2) C(3) P\left(\tree{\bf x}{1 \, 2 \, 3} \right) + C(1) \sum_{a \ne 1,2,3} P \left(\tree{\bf x}{1 \,a \, 3} \right) \Bigg]\notag \\
& =  \frac{1}{C(n+3) D(n+2)}\Bigg[\sum_{i=1}^n C(2i + 3) P^*(\widehat{\bf x}_i) P(123) + \underbrace{((q-2)C(3) - C(5))}_{ = C(1) \, \mbox{\tiny by } \eqref{eq:HL0}}P\left(\tree{\bf x}{1 \, 2 \, 3} \right)\notag \\
& \qquad \qquad \qquad \qquad \qquad \qquad + C(1) \sum_{a \ne 1,2,3} P \left(\tree{\bf x}{1 \,a \, 3} \right) \Bigg] \notag \\
& = \frac{1}{C(n+3) D(n+2)} \Bigg[\sum_{i=1}^n C(2i + 3) P^*(\widehat{\bf x}_i) P(123) + C(1) P^*({\bf x}) \underbrace{P(1) P(3)}_{= (q-2) P(123)}  \Bigg].
\end{align}

Now using \eqref{eq:HL2} with $j = 3$, $k = 2i$ and $\ell = n - 2i +1$, we obtain
\begin{equation*}
C(2i+3) D(n+1) + C(n-2i+1) D(3) = C(2i) D(n+4),
\end{equation*}
which implies that
\begin{align}
\label{eq:58}
\sum_{i = 1}^n C(2i+3) P^*(\widehat{\bf x}_i) &= \frac{1}{D(n+1)} \Bigg[D(n+4) \sum_{i = 1}^n C(2i) P^*(\widehat{\bf x}_i)  - D(3) \sum_{i = 1}^n C(n - 2i + 1) P^*(\widehat{\bf x}_i)  \Bigg] \notag \\
& = \frac{1}{D(n+1)} \Bigg[D(n+1)C(n+1)D(n+4) P^*({\bf x}) - D(n+1) D(n+3)P^*({\bf x}) \Bigg] \notag  \\
& = (C(n+1)D(n+4) - D(3)) P^*({\bf x}),
\end{align}
where the first term in the second equality is due to \eqref{eq:HL3}.
Further by \eqref{eq:HL1}, we have 
\begin{align}
\label{eq:59}
C(n+1)D(n+4) - D(3) &= \frac{1}{2}(D(2n+5) + D(3)) - D(3) \notag \\
& = \frac{1}{2}(D(2n+5) - D(3)) = C(n+4) D(n+1).
\end{align}
Thus,
$\sum_{a \ne 1,2}P \left(\btree{\bf x}{a}{1 \,2 \, 3} \right) = \frac{C(n+4)D(n+1) + (q-2) C(1)}{C(n+3) D(n+2)} P^*({\bf x}) P(123)$ by \eqref{eq:57}, \eqref{eq:58} and \eqref{eq:59}.
Again using \eqref{eq:HL2} with $j = 1$, $k = n+3$ and $\ell = -2$ leads to the desired result.

\smallskip
{\em Case 2:} $x_1 = 2$. We have:
\begin{align}
\label{eq:510}
& \sum_{a \ne 2} P \left(\btree{\bf x}{a}{1 \,2 \, 3} \right) \notag \\
& = \frac{1}{C(n+3)D(n+2)} \Bigg[ \sum_{i = 1}^n C(2i+3) \sum_{a \ne 2} P\left(\btree{\widehat{\bf x}_i}{a}{1 \, 2 \, 3} \right) + C(1) \sum_{a \ne 1,2,3} P\left(\tree{\bf x}{1 \, 2 \, 3} \right)  \Bigg] \notag \\
& = \frac{1}{C(n+3)D(n+2)} \Bigg[ \sum_{i = 1}^n C(2i+3) P^*(\widehat{\bf x}_i)P(123) + C(1) P^*({\bf x}) \underbrace{P(1) P(3)}_{= (q-2) P(123)} \Bigg].
\end{align}
Note that the equation \eqref{eq:510} reduces to \eqref{eq:57}, and the rest follows exactly the same as Case 1.

\smallskip
{\em Case 3:} $x_1 \ne 1,2,3$. Assume w.l.o.g. that $x_1 = 4$.
We then have:
\begin{align}
\label{eq:511}
& \sum_{a \ne 2, 4} P \left(\btree{\bf x}{a}{1 \,2 \, 3} \right) \notag \\
& = \frac{1}{C(n+3)D(n+2)} \Bigg[ \sum_{i = 2}^n C(2i+3) \sum_{a \ne 2,4} P\left(\btree{\widehat{\bf x}_i}{a}{1 \, 2 \, 3} \right) + C(5) \sum_{a \ne 2,4} P\left(\btree{\widehat{\bf x}_1}{a}{1 \, 2 \, 3} \right)  \notag \\
& \qquad \qquad \qquad \qquad \qquad + (q-2) C(3) P\left( \tree{\bf x}{1 \, 2 \, 3}\right) + C(1) \sum_{a = 5}^q P\left(\tree{\bf x}{1 \, a \, 3} \right)  \Bigg] \notag \\
& = \frac{1}{C(n+3)D(n+2)} \Bigg[ \sum_{i = 2}^n C(2i+3) P^*(\widehat{\bf x}_i)P(123) + C(5) \left(P^*(\widehat{\bf x}_1)P(123) - P \left( \tree{\bf x}{1 \, 2 \, 3}\right)\right) \notag \\
& \qquad \qquad \qquad \qquad \qquad + ((q-2)C(3) - C(5))P \left( \tree{\bf x}{1 \, 2 \, 3}\right) + C(1) \sum_{a = 5}^q P\left(\tree{\bf x}{1 \, a \, 3} \right) \Bigg] \notag \\
& = \frac{1}{C(n+3)D(n+2)} \Bigg[ \sum_{i = 2}^n C(2i+3) P^*(\widehat{\bf x}_i)P(123) + C(1) P^*({\bf x}) P(1)P(3) \Bigg].
\end{align}

The equation \eqref{eq:511} reduces to \eqref{eq:57}, and the remaining again follows Case 1.

\smallskip
($v)$
There are two cases.

\smallskip
{\em Case 1:} $x_1 \ne 1,2$. Assume w.l.o.g. that $x_1 = 3$. 
We have:
\begin{align}
\label{eq:512}
\sum_{a \ne 1,2,3} P\left( \tree{\bf x}{1 \, a \, 2}\right) & = \frac{1}{C(n+2)D(n+1)} \Bigg[ \sum_{i = 2}^n C(2i+1) \sum_{a \ne 1,2,3} P\left(\tree{\widehat{\bf x}_i}{1 \, a \, 2} \right) \notag \\
& \qquad \qquad \qquad + C(3) \sum_{a \ne 1,2,3} P\left(\tree{\widehat{\bf x}_1}{1 \, a \, 2} \right) + (q-3)C(1)P\left( \tree{\widehat{\bf x}_1}{1 \, x_1 \, 2}\right) \Bigg] \notag \\
& = \frac{1}{C(n+2)D(n+1)} \Bigg[\sum_{i = 2}^n C(2i+1) P^*(\widehat{\bf x}_i)P^*(1)P^*(2) \notag \\
& \qquad + C(3) \left(P^*(\widehat{\bf x}_1)P^*(1)P^*(2) - P\left(\tree{\widehat{\bf x}_i}{1 \, a \, 2} \right) \right) + (q-3)C(1)P\left( \tree{\widehat{\bf x}_1}{1 \, x_1 \, 2}\right) \Bigg] \notag \\
& = \frac{1}{C(n+2)D(n+1)}\Bigg[\sum_{i = 1}^n C(2i+1) P^*(\widehat{\bf x}_i)P^*(1)P^*(2) \notag \\
&  \quad \qquad \qquad \qquad \qquad + \underbrace{((q-3)C(1) - C(3))}_{ = 0} P\left( \tree{\widehat{\bf x}_1}{1 \, x_1 \, 2}\right)  \Bigg].
\end{align}

Now using \eqref{eq:HL2} with $j = 1$, $k = 2i$ and $\ell = n - 2i + 1$, we get
\begin{equation*}
C(2i + 1) = \frac{1}{D(n+1)} (C(2i) D(n+2) - C(n-2i +1) D(1)).
\end{equation*}
This implies that
\begin{align}
\label{eq:513}
\sum_{i = 1}^n C(2i+1) P^*(\widehat{\bf x}_i) & = \frac{1}{D(n+1)}\Bigg[D(n+2) \underbrace{\sum_{i = 1}^n C(2i) P^*(\widehat{\bf x}_i)}_{= C(n+1)D(n+1)P^*(\bf{x})}
- D(1) \underbrace{\sum_{i = 1}^n C(n-2i+1)P^*(\widehat{\bf x}_i)}_{= D(n+1)P^*(\bf{x})} \Bigg] \notag \\
& = (C(n+1) D(n+2) - D(1)) P^*(\bf{x}) \notag \\
& = \left[\frac{1}{2}(D(2n+3) + D(1)) - D(1) \right] P^*({\bf x}) =  C(n+2) D(n+1) P^*({\bf x}).
\end{align}
Combining \eqref{eq:512} and \eqref{eq:513} yields the desired result.

\smallskip
{\em Case 2:} $x_1 = 1$ or $2$. Assume w.l.o.g. that $x_1 = 1$. 
We then get
\begin{align}
\label{eq:514}
\sum_{a \ne 1,2,3} P\left( \tree{\bf x}{1 \, a \, 2}\right) &= \frac{1}{C(n+2)D(n+1)} \sum_{i = 1}^n C(2i+1) \sum_{a \ne 1,2} P\left(\tree{\widehat{\bf x}_i}{1 \, a \, 2} \right) \notag \\
& = \frac{1}{C(n+2)D(n+1)} \sum_{i = 1}^n C(2i+1) P^*(\widehat{\bf x}_i)P^*(1)P^*(2).
\end{align}
The equation \eqref{eq:514} reduces to \eqref{eq:512}, and the rest follows the same as Case 1.
\end{proof}

\subsection{Nonnegativity conditions}
\label{sc52}
Now we prove that for $q \ge 5$, $P(\cdot)$ given by \eqref{eq:DCR2} defines a probability measure on the proper colorings of $\mathscr{S}^3_{(1,1,\infty)}$, or equivalently each term of the family $P(\cdot)$ is nonnegative.
By the deletion-concatenation recursion \eqref{eq:DCR21}, the terms $P\left( \tree{\bf x}{x_L \, x_0 \, x_R}\right)$ for $x_L \ne x_R$ are automatically nonnegative. 
So it remains to check the case $x_L = x_R$.
Again by symmetry of the colors, we consider $P\left(\tree{\bf x}{1 \, 2 \, 1} \right)$, with ${\bf x} = (x_1, \ldots, x_n)$ for some $n$.

The key idea is to find a sequence $(\alpha_n; \, n \in \mathbb{N}) \in (0,1)^{\mathbb{N}}$ such that 
for each coloring $\tree{\bf x}{1 \, 2 \, 1}$ with ${\bf x}$ of length $n$,
\begin{equation}
\label{eq:nnest}
P^*(12{\bf x}) - P \left(\tree{\bf x}{1 \, 2 \, 1} \right) \le \alpha_n  P^*(12\bf{x}),
\end{equation}
which implies the nonnegativity of $P\left(\tree{\bf x}{1 \, 2 \, 1}\right)$.
The sequence $(\alpha_n; \, n \in \mathbb{N})$ is defined by recursion to match the induction.
The following lemma gives the inductive step.

\begin{lemma}
\label{lem:induction}
Assume that the inequality \eqref{eq:nnest} holds for all ${\bf x}$ of length $n-1$, with $\alpha_{n-1} > 0$.
Then for each coloring $\tree{\bf x}{1 \, 2 \, 1}$ with ${\bf x}$ of length $n$, we have
\begin{equation}
\label{eq:induction}
P^*(12{\bf x}) - P \left(\tree{\bf x}{1 \, 2 \, 1} \right)  \le \alpha_{n-1} P^*(12{\bf x}) + \frac{ \sqrt{q} \, \alpha_{n-1}}{2 C(n+2) D(n+1)} \left(q  P^*(12{\bf x}) - P^*(2{\bf x})\right).
\end{equation}
\end{lemma}

\begin{proof}
First we claim that 
\begin{equation}
\label{eq:516}
P^*(12{\bf x}) - P\left(\tree{\bf x}{1 \, 2 \, 1} \right) \le \frac{\alpha_{n-1}}{C(n+2)D(n+1)} \left(C(1) P^*(1{\bf x}) + \sum_{i = 1}^n  C(2i + 1) P^*(12 \widehat{\bf x}_i)\right)
\end{equation}
We distinguish two cases: $x_1 = 1$ and $x_1 \ne 1$.

\smallskip
{\em Case 1:} $x_1 = 1$. We have
\begin{align*}
P^*(12{\bf x}) - P\left(\tree{\bf x}{1 \, 2 \, 1} \right) &= \sum_{a \ne 1,2} P\left(\tree{\bf x}{1 \, 2 \, a} \right) \\
& = \sum_{a \ne 1,2} \frac{1}{C(n+2)D(n+1)} \sum_{i = 1}^n C(2i+1) P\left(\tree{\widehat{\bf x}_i}{1 \, 2 \, a} \right) \\
& = \frac{1}{C(n+2)D(n+1)} \sum_{i = 1}^n C(2i+1) \left(P^*(12\widehat{\bf x}_i) - P\left(\tree{\widehat{\bf x}_i}{1 \, 2 \, 1} \right) \right) \\
& \le \frac{\alpha_{n-1}}{C(n+2)D(n+1)} \sum_{i = 1}^n C(2i+1)P^*(12\widehat{\bf x}_i),
\end{align*}
where the second equality is due to the recursion \eqref{eq:DCR21}, and the last inequality is by the induction.
So the inequality \eqref{eq:516} holds since $P(1{\bf x}) = 0$ in this case.

\smallskip
{\em Case 2:} $x_1 \ne 1$. Similar to Case 1, we obtain
\begin{align*}
& P^*(12{\bf x}) - P\left(\tree{\bf x}{1 \, 2 \, 1} \right) \\
& = \sum_{a \ne 1,2} \frac{1}{C(n+2)D(n+1)} \left(\sum_{i = 1}^n C(2i+1) P\left(\tree{\widehat{\bf x}_i}{1 \, 2 \, a} \right) + C(1) 
P\left(\tree{\widehat{\bf x}_1}{1 \, x_1 \, a} \right) \right) \\
& = \frac{1}{C(n+2)D(n+1)} \Bigg[ \sum_{i = 1}^n C(2i+1) \left(P^*(12\widehat{\bf x}_i) - P\left(\tree{\widehat{\bf x}_i}{1 \, 2 \, 1} \right) \right) \\
& \qquad \qquad \qquad \qquad \qquad + C(1) \left(P^*(1{\bf x}) - P\left(\tree{\widehat{\bf x}_1}{1 \, x_1 \, 1}\right) - P\left(\tree{\widehat{\bf x}_1}{1 \, x_1 \, 2}  \right) \right) \Bigg] \\
& \le \frac{\alpha_{n-1}}{C(n+2)D(n+1)}\left( C(1)P^*(1{\bf x}) + \sum_{i = 1}^n C(2i+1)P^*(12\widehat{\bf x}_i) \right).
\end{align*}

Now by \eqref{eq:HL3}, we have
\begin{equation}
\label{eq:517}
C(2) P^*(2{\bf x}) + C(4) P^*(1{\bf x}) + \sum_{i = 1}^n C(2i+4) P^*(12\widehat{\bf x}_i) = C(n+3)D(n+3)P^*(12{\bf x}).
\end{equation}
Further using \eqref{eq:HL2} with $j = -3$, $k = 2i+4$ and $\ell = n - 2i -1$, we get
\begin{equation}
\label{eq:518}
C(2i+1) D(n+3) = C(2i+4) D(n) + C(n-2i-1) D(3).
\end{equation}
Combining \eqref{eq:517} and \eqref{eq:518} yields
\begin{align}
\label{eq:519}
& C(1)P^*(1{\bf x}) + \sum_{i = 1}^n C(2i+1) P^*(12\widehat{\bf x}_i) \notag \\
& = \frac{1}{D(n+3)} \Bigg[  D(n) \left(C(n+3)D(n+3)P^*(12{\bf x}) -  C(2) P^*(2{\bf x})  \right)\notag \\
&  \qquad \qquad \qquad \quad + D(3) (D(n+3) P^*(12{\bf x}) - C(n+1) P^*(2{\bf x})) \Bigg] \notag \\
& = (D(n)C(n+3) + D(3)) P^*(12{\bf x}) - \frac{C(2)D(n) + C(n+1) D(3)}{D(n+3)} P^*(2{\bf x}).
\end{align}
Now using \eqref{eq:HL2} with $j = 3$, $k = n$ and $\ell = 0$, we get
$C(n+3) D(n) + D(3) = C(n) D(n+3)$. 
Also using \eqref{eq:HL2} with $j = n$, $k = 1$ and $\ell = 2$,  we get
$C(n+1) D(3) + C(2) D(n) = C(1) D(n+3)$. 
Injecting these equalities into \eqref{eq:519}, we obtain
\begin{equation}
\label{eq:520}
C(1)P^*(1{\bf x}) + \sum_{i = 1}^n C(2i+1) P^*(12\widehat{\bf x}_i) = C(n) D(n+3) P^*(12{\bf x}) - C(1) P^*(2{\bf x}).
\end{equation}
Injecting \eqref{eq:520} into \eqref{eq:516} gives
\begin{equation*}
P^*(12{\bf x}) - P\left(\tree{\bf x}{1 \, 2 \, 1} \right) \le \frac{\alpha_{n-1}}{C(n+2) D(n+1)}  \left(  C(n) D(n+3) P^*(12{\bf x}) - C(1) P^*(2{\bf x}) \right).
\end{equation*}
Finally, using \eqref{eq:HL2} with $j = 2$, $k = n$ and $\ell = 1$, we get
$C(n) D(n+3) = (n+2)D(n+1) C(1) D(2)$, which leads to the desired result.
\end{proof}

In view of \eqref{eq:induction}, the goal is to bound $q P(12{\bf x}) - P(2{\bf x})$, or equivalently to bound $P(12{\bf x})$ in terms of $P(2{\bf x})$. 
The following proposition gives such a bound.

\begin{proposition}
\label{prop:boundplus}
Let $q \ge 4$. For each $q$-coloring ${\bf x}$ of a finite length, we have:
\begin{equation}
\label{eq:impbound}
P^*(a{\bf x}) \le \frac{2}{q+ \sqrt{q(q-4)}} P^*({\bf x}).
\end{equation}
\end{proposition}

\begin{proof}
The idea is again to find a sequence $(\beta_n; \, n \in \mathbb{N})$ inductively, which satisfies 
$P^*(a{\bf x}) \le \beta_n P^*({\bf x})$.
For ${\bf x}$ a coloring of length $n$, we have
\begin{align}
\label{eq:523}
P^*(a{\bf x}) &= \frac{1}{D(n+2)} \left(C(n) P^*({\bf x}) + \sum_{i= 1}^n C(n-2i) P^*(a\widehat{\bf x}_i)   \right) \notag \\
& \le \frac{1}{D(n+2)} \left(C(n) P^*({\bf x}) + \beta_{n-1}\sum_{i= 1}^n C(n-2i) P^*(\widehat{\bf x}_i)   \right)
\end{align}
where the second inequality is obtained by induction.
Now using \eqref{eq:HL2} with $j = 1$, $k = n-2i$ and $\ell = 2i$, we get
\begin{equation*}
C(n - 2i) = \frac{1}{D(n+1)} (C(n-2i + 1) D(n) + C(2i) D(1)),
\end{equation*}
which implies that
\begin{align}
\label{eq:524}
& \sum_{i= 1}^n C(n-2i) P^*(\widehat{\bf x}_i)  \notag\\
& = \frac{1}{D(n+1)} \Bigg[D(n) \underbrace{\sum_{i = 1}^n C(n-2i+1) P^*(\widehat{\bf x}_i)}_{= D(n+1) P^*({\bf x})} + D(1) \underbrace{\sum_{i = 1}^n C(2i) P^*(\widehat{\bf x}_i)}_{= C(n+1) D(n+1) P^*({\bf x})} \Bigg] \notag\\
& = (D(n) + C(n+1)D(1)) P^*({\bf x}) \notag\\
& = (D(n+2) - C(n+1) D(1)) P^*({\bf x}).
\end{align}
Combining \eqref{eq:523} and \eqref{eq:524} yields
\begin{equation*}
P^*(a{\bf x}) \le \left( \beta_{n-1} + \frac{C(n) - \beta_{n-1}\sqrt{q}C(n+1)}{D(n+2)}\right) P^*({\bf x}).
\end{equation*}
It suffices to take
\begin{equation*}
\beta_n = \beta_{n-1} + \frac{C(n) - \beta_{n-1}\sqrt{q}C(n+1)}{D(n+2)}, \quad \mbox{with } \beta_0 = \frac{1}{q}.
\end{equation*}
Elementary analysis shows that for $q > 4$, 
$\beta_n \uparrow \frac{2}{q+ \sqrt{q(q-4)}}$ as $n \rightarrow \infty$, where $\frac{q+ \sqrt{q(q-4)}}{2}$ is the limit of the sequence $\sqrt{q}C(n+1)/C(n)$. 
The result follows immediately.
\end{proof}

It is a consequence of Proposition \ref{prop:boundplus} that $P^*(2{\bf x}) \ge \frac{q+ \sqrt{q(q-4)}}{2} P^*(12{\bf x})$.
Together with Lemma \ref{lem:induction}, we get the inductive estimate:
\begin{equation*}
P^*(12{\bf x}) - P \left(\tree{\bf x}{1 \, 2 \, 1} \right) \le \alpha_{n-1} \left(1 + \frac{q}{(\sqrt{q} + \sqrt{q-4})C(n+2) D(n+1)} \right) P^*(12\bf{x}).
\end{equation*}
So we take
\begin{equation}
\label{eq:alpharec}
\alpha_n = \alpha_{n-1} \left(1 + \frac{q}{(\sqrt{q}+\sqrt{q-4})C(n+2) D(n+1)} \right), \quad \mbox{with } \alpha_0 = \frac{q-1}{q}.
\end{equation}
It is clear that the sequence $(\alpha_n; \, n \in \mathbb{N})$ defined by \eqref{eq:alpharec} is increasing.
To conclude, we need to show that the limit of $(\alpha_n; \, n \in \mathbb{N})$ is bounded from above by $1$. 
That is,
\begin{equation}
\label{eq:526}
\frac{q-1}{q} \prod_{n = 1}^{\infty} \left(1 + \frac{q}{(\sqrt{q} +\sqrt{q-4})C(n+2) D(n+1)} \right) \le 1
\end{equation}
Recall the definition of $C(n)$, $D(n)$ from \eqref{eq:CD}.
It is not hard to see that for each $q \ge 5$ and $n \ge 1$,
\begin{align*}
& C(n) \ge \frac{1}{2} \left(\frac{\sqrt{q} + \sqrt{q-4}}{2} \right)^n \ge \frac{1}{2} (q - 2.4)^\frac{n}{2}, \\
& D(n) \ge \left(\frac{\sqrt{q} + \sqrt{q-4}}{2} \right)^n \ge (q - 2.4)^\frac{n}{2}.
\end{align*}
Thus, for $q \ge 6$,
\begin{align*}
\prod_{n = 1}^{\infty} \left(1 + \frac{q}{(\sqrt{q} +\sqrt{q-4})C(n+2) D(n+1)} \right) & \le \prod_{n = 1}^{\infty} \left( 1 + \frac{q}{(q-2.4)^2}\, \frac{1}{(q - 2.4)^n} \right) \\
& \le \exp\left(\sum_{n = 1}^{\infty} \frac{6}{3.6^2} \, \frac{1}{(q - 2.4)^n} \right) \\
& = \exp \left(\frac{6}{3.6^2 (q-3.4)} \right).
\end{align*}
It is easily checked that $\max_{x \ge 6} \frac{x-1}{x}\exp \left(\frac{6}{3.6^2 (x-3.4)} \right) < 1$, 
see Figure \ref{fig:plot} for a plot of the function $x \to \frac{x-1}{x}\exp \left(\frac{6}{3.6^2 (x-3.4)} \right)$.

\begin{figure}[h]
\label{fig:plot}
\includegraphics[width=0.45 \textwidth]{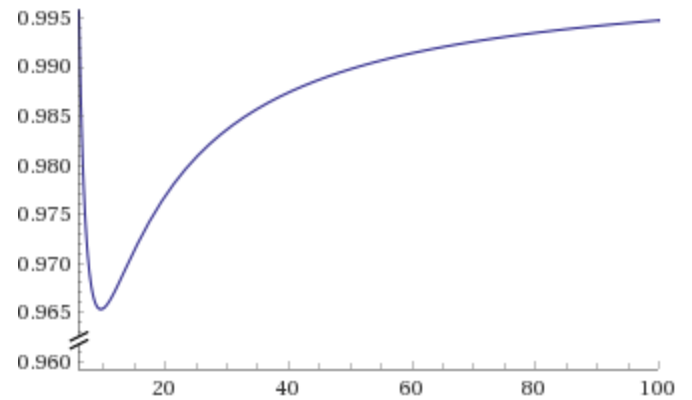}
\caption{Plot of $x \to \frac{x-1}{x}\exp \left(\frac{6}{3.6^2 (x-3.4)} \right)$ on $[6, 100]$}
\end{figure}

It remains to prove the inequality \eqref{eq:526} for $q = 5$. 
In fact, there is a tighter bound $D(n) \ge 2 \cdot (5 - 2.4)^{\frac{n}{2}}$ for $q = 5$ and $n \ge 2$. 
Thus, we get
\begin{align*}
\prod_{n = 1}^{\infty} \left(1 + \frac{5}{(\sqrt{5}+1) C(n+2) D(n+1)} \right)  & \le 
\left(1 + \frac{1}{5 + \sqrt{5}} \right)\prod_{n = 2}^{\infty} \left( 1 + \frac{2.5}{(5-2.4)^{n+2}} \right) \\
& \approx 1.242 < \frac{5}{4},
\end{align*}
which leads to the desired result.

\section{Open problems}
\label{sc6}

In this section, we give a few open problems related to finitely dependent colorings of star graphs
which complement those in \cite{HHL20, HL15, HL16}.

\begin{enumerate}[itemsep = 3 pt]
\item
Prove or disprove the statement that there is no $1$-dependent $4$-coloring of $\mathscr{S}^3_{(m,n,\infty)}$ for any $m, n > 0$ which is invariant under permutations of the colors, and whose restriction to any path is a symmetric $1$-dependent $4$-coloring.
This problem is related to the uniqueness of symmetric $1$-dependent colorings of $\mathbb{Z}$.
\item
Is it possible to give a probabilistic construction of the $1$-dependent coloring \eqref{eq:DCR2} of $\mathscr{S}^3_{(1,1,\infty)}$ in a similar spirit to that in \cite{HHL20} ?
\item
Is it possible to construct a $1$-dependent $q$-coloring for any $q$ of $\mathscr{S}^3_{(m,n,\infty)}$ for $m > 1$ and $n \ge 1$, which satisfies the conditions $(i)$-$(ii)$ ?
 \item
Is it possible to construct a symmetric $k$-dependent $q$-coloring for any $k, q$ of $\mathscr{S}^d$ ?
\end{enumerate}

Since the star graphs are building components of many other graphs, answering any of the above questions will enhance our understandings of finitely dependent colorings of general graphs.

\begin{acks}[Acknowledgments]
The authors thank Alexander Holroyd for helpful discussions.
The authors also thank an anonymous reviewer for constructive suggestions which improve the presentation of the paper.
\end{acks}

\begin{funding}
The second author was supported by NSF Grants DMS-2113779 and DMS-2206038, and by a start-up grant at Columbia University.
\end{funding}



\bibliographystyle{imsart-number} 
\bibliography{unique}       

\end{document}